\documentclass[12pt, a4paper]{amsart}
\usepackage[unicode]{hyperref}
\usepackage{xcolor}
\hypersetup{
    colorlinks,bookmarksopen,bookmarksnumbered,pdfstartview=FitH,
    linkcolor={red!50!black},
    citecolor={blue!50!black},
    urlcolor={blue!80!black}
}

\usepackage[english]{babel}
\usepackage{amssymb,amsmath,amsthm}
\usepackage{verbatim}

\usepackage[all]{xy}

\numberwithin{equation}{section}

\newtheorem{dummy}{dummy}[section]

\newtheorem{theorem}[dummy]{Theorem}
\newtheorem{corollary}[dummy]{Corollary}
\newtheorem{lemma}[dummy]{Lemma}
\newtheorem{proposition}[dummy]{Proposition}
\newtheorem{remark}[dummy]{Remark}
\newtheorem{example}[dummy]{Example}
\newtheorem{conjecture}[dummy]{Conjecture}
\newtheorem{definition}[dummy]{Definition}

\def\A{\mathbb A}
\def\C{\mathbb C}
\def\F{\mathbb F}
\def\L{\mathbb L}
\def\P{\mathbb P}
\def\Q{\mathbb Q}
\def\R{\mathbb R}
\def\S{\mathbb S}

\def\Z{\mathbb Z}
\def\AA{\mathcal A}

\def\HH{\mathcal H}

\def\LL{\mathcal L}
\def\MM{\mathcal M}
\def\OO{\mathcal O}
\def\TT{\mathcal T}

\def\wt{\widetilde}

\def\={\;=\;}
\def\bal{\begin{aligned}}
\def\eal{\end{aligned}}
\def\be{\begin{equation}\label}
\def\ee{\end{equation}}

\newcommand \zhenya {\color{blue}}

\newcommand{\Sym}[1] {{#1}^{(2)}}
\newcommand{\Hilb}[1] {{#1}^{[2]}}

\title[Fano variety and rationality for a cubic]{The Fano variety of lines and \\rationality problem for a cubic hypersurface}
\author{Sergey Galkin, Evgeny Shinder}
\begin{document}
\begin{abstract}
We find a relation between a cubic hypersurface $Y$ and its Fano variety of lines $F(Y)$
in the Grothendieck ring of varieties. We use this relation to study the Hodge structure of $F(Y)$.
Finally we propose a criterion for rationality of a smooth cubic hypersurface
in terms of its variety of lines.

In particular, we show that if the class of the affine line 
is not a zero-divisor in the Grothendieck ring,
then the variety of lines on a smooth rational cubic fourfold
is birational to a Hilbert scheme of two points on a $K3$ surface.
\end{abstract}
\maketitle
\tableofcontents 
\section{Introduction}

Let $Y \subset \P^{d+1}$ be a 
cubic hypersurface
over a field $k$.
The Fano variety $F(Y)$ of lines on $Y$ is defined as the closed subvariety 
of the Grassmannian $Gr(1, \P^{d+1}) = Gr(2, d+2)$ consisting of lines
$L \subset Y$. See \cite{BV}, \cite{AK} for details and Section \ref{subsec-fano-def}
for a short summary on the Fano variety of lines on a cubic.

\medskip

For the purpose of this Introduction assume that $Y$ is smooth, in which case $F(Y)$ is smooth
projective of dimension $2d-4$, and connected for $d \ge 3$. Consider the low-dimensional cases:

\begin{itemize}

\item $d=2$: $F(Y)$ consists of $27$ isolated points. 
This has been discovered in the correspondence between Cayley and Salmon 
and has been published in 1849 \cite{Cay, Sal}.

\item $d=3$: $F(Y)$ is a surface of general type
which has been introduced 
by Fano \cite{F}, and then studied by
Bombieri and Swinnerton-Dyer \cite{BS} in their
proof of the Weil conjectures for a smooth cubic threefold 
defined over a finite field and
by Clemens and Griffiths \cite{CG} in their proof
of irrationality of a smooth complex cubic threefold.

\item $d=4$: $F(Y)$ is a holomorphic symplectic fourfold
\cite{BD}. For several types of cubic fourfolds $F(Y)$ is
isomorphic to a Hilbert scheme of two points on a $K3$ surface 
\cite{BD}, \cite{Has2}.
\end{itemize}

\medskip

In this paper we study the geometry of $F(Y)$ in terms of the geometry of $Y$
and discuss applications to rationality of $Y$.
In the most concise form the relation between the geometry of the Fano variety $F(Y)$
and the cubic $Y$, which we call the $Y$-$F(Y)$ relation looks like:
\[
[\Hilb{Y}] = [\P^d][Y] + \L^2[F(Y)]. 
\]

Here $\Hilb{Y}$ is the Hilbert scheme of length two subschemes on $Y$.
The equality holds in the Grothendieck ring of varieties, and encodes the basic
fact that a line $L$ intersecting the cubic $Y$ in two points determines the third
intersection point unless $L \subset Y$.
See Theorems \ref{main-thm1}, \ref{main-thm2} for different ways of expressing the
$Y$-$F(Y)$ relation.

Different kinds of invariants of the Fano variety may be computed using the $Y$-$F(Y)$
relation.
For example we can easily compute the number of lines on real and complex smooth or singular
cubic surfaces (see Examples \ref{ex-duVal}, \ref{ex-real}).
On the other hand we can compute the Hodge structure of $H^*(F(Y), \Q)$ for a
smooth complex cubic $Y$ of an arbitrary dimension $d$.
It turns out that the Hodge structure of $F(Y)$ is essentially 
the symmetric square of the Hodge structure of $Y$  (Theorem \ref{thm-hodge}).

\medskip

A central question in studying cubics is that of rationality of a smooth cubic $d$-fold.
This question is highly non-trivial in dimension $d \ge 3$ (for $d\ge 2$ if the base field 
is not algebraically closed); at the moment the answer is not known for $d \ge 4$.

As already mentioned above Clemens and Griffiths used delicate analysis of the Fano variety
of lines in relation to the intermediate Jacobian of the cubic threefold to prove irrationality of all smooth
complex cubic threefolds \cite{CG}. 
Murre uses similar techniques to prove irrationality of cubic threefolds in characteristic $p \ne 2$
\cite{Mur72, Mur73, Mur74}.

In dimension $4$ the situation is much more complicated. 
Let us give a brief overview on rationality of smooth cubic fourfolds over $k=\C$.
First of all, there are examples of smooth rational cubic fourfolds: the simplest
ones are cubic fourfolds containing two disjoint $2$-planes 
and Pfaffian cubic fourfolds.
These and other classes of rational cubic fourfolds have been studied by
Morin \cite{Mor}, Fano \cite{F43}, Tregub \cite{Tr1}, \cite{Tr2}, Beauville-Donagi \cite{BD} and Hassett \cite{Has1}.

Nevertheless a very general cubic fourfold is expected to be irrational.
More precisely, according to a conjecture made by Iskovskih in the early 1980s (see \cite{Tr1}),
the algebraic cycles inside $H^4(Y, \Q)$ of a rational cubic fourfold
form a lattice of rank at least two.
It is known that such cubic fourfolds form a countable union of divisors in the
moduli space. Thus for a very general cubic fourfold algebraic classes in $H^4(Y, \Q)$ are
multiples of $h^2$, where $h \in H^2(Y, \Q)$ is the class of the hyperplane section \cite{Zar, Has2}

In light of this discussion it is quite remarkable that no irrational cubic fourfold is known at the moment.
Kulikov deduced irrationality of a general cubic fourfold from a certain conjectural indecomposability of
Hodge structure of surfaces \cite{Kul}; the latter version of indecomposability however was recently shown
to be false \cite{ABB}.

Hassett called cubic fourfolds with an extra algebraic class in $H^4(Y, \Q)$ \emph{special},
and studied them in detail, giving
complete classification of special cubic fourfolds into a countable union of divisors inside
the moduli space of all cubic fourfolds \cite{Has2}.
Hassett also classified those special cubic fourfolds $Y$ that have an associated $K3$ surface $S$.
This basically means that the primitive cohomology lattices of $Y$ and $S$ are isomorphic
(see Definition \ref{def-assK3} for details).

It is expected that rational smooth cubic fourfolds are not only special, but also
have associated $K3$ surfaces. 

Kuznetsov \cite{Kuz10} gave a conjectural criterion for rationality of a smooth cubic
fourfold based on derived categories. His condition also involves existence of a $K3$
surface $S$ associated to $Y$: the derived category of coherent sheaves on $S$
must be embedded into the derived category of coherent sheaves on $Y$.
Furthemore, Addington and Thomas showed that this criterion is generically 
equivalent to the Hodge-theoretic criterion of Hassett \cite{AT}.

A different but related conjectural necessary condition for rationality of cubic
fourfolds is given by Shen: according to \cite[Conjecture 1.6]{Sh}
a smooth rational cubic fourfold $Y$, $F(Y)$ must have a certain algebraic class in the 
middle cohomology $H^4(F(Y), \Z)$.


\medskip

In any even dimension $d=2r$ there exist smooth rational cubics.
To construct one we start with an $r$-dimensional subvariety $W \subset \P^{2r+1}$ with \emph{one apparent double point}.
This means that through a general point $p \in \P^{2r+1}$ there is a unique secant line $L_p$ to $W$, 
i.e. a line which intersects $W$ in two points.
A simple geometric construction which has been used by Morin \cite{Mor} and Fano \cite{F43}
shows that any cubic $Y$ containing such a $W$ is rational \cite[Prop.9]{Rus}.
The case of $2r$-dimensional cubics containing two disjoint $r$-planes is
a particular case of this construction.

Smooth connected varieties with one apparent double point exist for any dimension $r$ \cite{Bab31, Ed32, Rus}, and
such varieties are classified in small dimensions. For example, for $r=1$ 
it is the twisted cubic in $\P^3$, and for $r=2$ there are two degree four rational normal scrolls and
a del Pezzo surface of degree five \cite{Sev, Rus}.

We finish this overview of known results on rationality of cubics by noting that there is no examples of smooth rational 
cubics of odd dimension.

\medskip

In the direction of irrationality of cubic fourfolds we prove the following (see Theorem \ref{thm-rat4}):
let $k$ be a field of characteristic zero and
assume the Cancellation Conjecture: $\L = [\A^1]$ is not a zero divisor in the Grothendieck ring $K_0(Var/k)$ of varieties.
If $Y$ is a smooth rational cubic fourfold, then $F(Y)$ is birationally equivalent
to $\Hilb{S}$ for some $K3$ surface $S$. In particular $Y$ is special and
$S$ is associated to $Y$ in the sense of Hodge structure (Proposition \ref{prop-indec4}).
According to the discussion above this in particular implies that a very general smooth complex cubic fourfold $Y$
is irrational.

Modulo the same assumption ($\L$ being not a zero-divisor) we give a shorter proof for the result
of Clemens and Griffiths on irrationality of smooth cubic threefolds (Theorem \ref{thm-rat3}).

We also get a criterion for irrationality of higher-dimensional smooth cubics: 
the Fano variety of lines on a smooth rational cubic must be \emph{stably decomposable} (see Definition \ref{def-decomp}
and Theorem \ref{thm-rat}). 
However it is not clear at the moment whether
this criterion gives an obstruction to rationality in dimension $d \ge 5$.

\medskip

Our approach to irrationality is based on a result due to Larsen and Lunts
\cite{LL} (see also \cite{Bit}), which we recall in Theorem \ref{thmLL}.
This result itself is based on the Weak Factorization Theorem \cite{W,AKMW}
thus we need to assume that $char(k) = 0$.
Roughly speaking the theorem of Larsen and Lunts says that in the quotient ring $K_0(Var)/(\L)$
each class has a unique decomposition into classes of stable birational equivalence.
The same sort of uniqueness lies in the heart of the proof of irrationality of cubic threefolds
by Clemens and Griffiths \cite{CG} who consider decompositions of the principally polarized intermediate Jacobian 
of a cubic threefold.

Starting with a rational smooth cubic $d$-fold $Y$ we write its class in the form
\[
[Y] = [\P^d] + \L \cdot \MM_Y,
\]
for some $\MM_Y \in K_0(Var/k)$ (Corollary \ref{cor-ratM}) and plug this into the $Y$-$F(Y)$ relation.
Assuming the Cancellation Conjecture we may divide by $\L^2$ and then deduce using the theorem of Larsen and Lunts
that $F(Y)$ is stably decomposable. 
This is not possible in dimension $d=3$ and in dimension $d=4$ yields a birational equivalence
between $F(Y)$ and $\Hilb{S}$.
Our approach is especially efficient in these two dimensions as the Fano of variety of lines
has non-negative Kodaira dimension for $d \le 4$, and for such varieties the notion of
stable birational equivalence coincides with birational equivalence
due to existence of the MRC fibration \cite{KMM, Kol-MRC, GHS} (Lemma \ref{lemma-mrc}).


\medskip

Let us now briefly explain the structure of the paper.
Sections \ref{sec-grothendieck}, \ref{sec-hilb}, \ref{sec-fano} contain material
on the Grothendieck ring of varieties, the Hilbert
scheme of length two subschemes and the Fano variety of lines on a cubic.
Most of this is well-known except possibly the discussion of decomposability of the Fano variety
of lines in \ref{subsec-decomp} and \ref{subsec-decomp4}.

In Section \ref{sec-formula} we prove several versions of the $Y$-$F(Y)$ relation for a possibly
singular cubic hypersurface over an arbitrary field and deduce simple consequences
of this relation.

In Section \ref{sec-hodge} we express the Hodge structure of $F(Y)$ with rational coefficients
in terms of the Hodge structure of $Y$ for a smooth complex cubic hypersurface. 
In particular this recovers known results for cubic threefolds \cite{CG} and fourfolds \cite{BD}.

Section \ref{sec-rational} contains applications of the $Y$-$F(Y)$ relation to rationality of cubics.

\medskip

We would like to thank our friends and colleagues 
Arend Bayer, Paolo Cascini, Sergey Finashin, Sergey Gorchinskiy, Alexander Kuznetsov,
Fran\c{c}ois Loeser, Yuri Prokhorov, Francesco Russo, Nick Shepherd-Barron, Constantin Shramov,
Evgeny Shustin, Nicolo Sibilla, Maxim Smirnov, Fedor Lazarevich Zak 
for discussions, references and their interest in our work.
Special thanks go to Constantin Shramov, Andrey Soldatenkov and Ziyu Zhang for
their comments on a draft of this paper.
The second named author is greatly indebted to Daniel Huybrechts' seminar
on cubic hypersurfaces in Bonn in the Summer semester of 2013 which has been a great
opportunity to learn about cubics.

%
%
%
%
%
%
%
%
%
%

\section{The Grothendieck ring of varieties}
\label{sec-grothendieck}

Detailed references on the Grothendieck ring of varieties are \cite{L,Bit,DL}.

\subsection{Generalities}

Throughout the paper we work in the Grothendieck ring $K_0(Var/k)$ of varieties over $k$, which as an abelian group
is generated by classes $[X]$ for quasi-projective varieties $X$ over $k$ with relations
\[
[X] = [U] + [Z]
\]
for any closed $Z \subset X$  with open complement $U$. $K_0(Var/k)$ becomes a ring with the product defined on generators
as
\[
[X] \cdot [Y] = [X \times Y]\footnote{If the field $k$ is not perfect we take the class of $X \times Y$ with the reduced
scheme structure.}.
\]
Note that $1 = [pt]$. 
We write $\L = [\A^1] \in K_0(Var/k)$ for the Lefschetz class.

\medskip

For each $n \ge 0$ the operations $X \mapsto Sym^n(X) = X^{(n)} := X^n / \Sigma_n$ descend to $K_0(Var/k)$ and satisfy
\be{symsum}
Sym^n(\alpha+\beta) = \sum_{i+j=n} Sym^i \alpha \cdot Sym^j \beta, \;\;\; \alpha, \beta \in K_0(Var/k).
\ee
In particular if for $m \ge 0$ we write $m \in K_0(Var/k)$ for the class of a disjoint union of $m$ points, one can prove that
\[
Sym^n(m) = \binom{n+m-1}{n} (= \dim_k Sym^n(k^m)), \; m \ge 0.
\]

We also have
\be{symL}
Sym^n(\L^m \alpha) = \L^{mn} Sym^n \alpha, \;\;\; \alpha \in K_0(Var/k)
\ee
(\cite[Lemma 4.4]{G}).

\medskip

The symmetric powers are put together in the definition of Kapranov's ``motivic'' zeta function:
\[
Z_{Kap}(X,t) = \sum_{n \ge 0} [Sym^n(X)] t^n \in K_0(Var/k)[[t]].
\]

\medskip

We will need the following two useful formulas in $K_0(Var/k)$:
\begin{itemize}
\item Let $X \to S$ be a Zarisky locally-trivial fibration with fiber $F$. Then
\be{proj-bun}
[X] = [F] \cdot [S]
\ee
This is proved by induction on the dimension of $S$.

\item Let $X$ to be a smooth variety and $Z \subset X$ be a smooth closed subvariety of codimension $c$.
Then
\be{blow-up}
[Bl_Z (X)] - [\P(N_{Z/X})] = [X] - [Z]
\ee
This follows from definitions. Note that by (\ref{proj-bun}) $[\P(N_{Z/X})] = [\P^{c-1}] [Z]$.

\end{itemize}

\medskip

If $char(k)=0$, then there is an alternative description of the Grothendieck ring $K_0(Var/k)$ 
due to Bittner: the generators are classes of smooth projective connected varieties and the relations
are of the form (\ref{blow-up}) \cite{Bit}.

\subsection{Realizations}

We will call a ring morphism $\mu$
from $K_0(Var/k)$ to any ring $R$ a \emph{realization homomorphism with values in $R$}.
We list some well-known examples of realization homomorphisms
together with realizations of the zeta-function.

\begin{itemize}

\item {\bf Counting points:} $k = \F_q$ is the finite field of $q$ elements, 
$\mu(X) = \#X(\F_q)  \in \Z$.
Then the realization 
\[
\#(Z_{Kap}(X,t)) = \prod_{x \in X_{0}} \frac1{1-t^{[k(x):k]}} = \exp \Bigl( \sum_{m \ge 1} \frac{\#X(\F_{q^m})}{m} t^m \Bigr),
\]
is the Hasse-Weil zeta-function
($X_0$ denotes the set of closed points of $X$).
In particular we have
\be{numSym2}
\#\Sym{X}(\F_q) = \frac{\#X(\F_q)^2 + \#X(\F_{q^2})}{2}.
\ee

\item {\bf Etale Euler characteristic:} $k$ any field, $\mu = \chi$, with $\chi(X) = \sum_{p \ge 0} (-1)^p \dim H^p_{et,c}(X_{\bar{k}}, \Q_l)  \in \Z$
being the geometric etale Euler characteristic with compact supports.
Here $l \ne char(k)$ is any prime. The standard comparison theorems imply that
if $k \subset \C$, then
$\chi(X) = \chi_\C(X) := \chi_c(X(\C))$.
We have
\[
\chi(Z_{Kap}(X,t)) = \Bigl(\frac1{1-t}\Bigr)^{\chi(X)},
\]
and in particular 
\be{chiSym2}
\chi(\Sym{X}) = \frac{\chi(X)(\chi(X)+1)}{2}.
\ee

\item {\bf Real Euler characteristic:} $k \subset \R$, $\mu = \chi_\R$ with $\chi_\R(X) = \chi_c(X(\R)) \in \Z$.
Then 
\[
\chi_\R(Z_{Kap}(X,t)) = \Bigl(\frac1{1-t^2}\Bigr)^{\frac{\chi_\C(X)-\chi_\R(X)}{2}}  
 \Bigl(\frac1{1-t}\Bigr)^{\chi_\R(X)} ,
\]
and in particular
\be{chiRSym2}
\chi_\R(\Sym{X}) = \frac{\chi_\R(X)^2  + \chi_\C(X)}{2}.
\ee

\item{\bf Hodge polynomials:} $k \subset \C$, 
\[
\mu([X]) = E(X_\C, u,v) = \sum_{p,q \ge 0} e_{p,q}(X_\C) u^p v^q \in \Z[u,v]\]
is the virtual Hodge-Deligne polynomial of $X_\C$. 
We have $e_{p,q}(X_\C) = (-1)^{p+q}h^{p,q}(X_\C)$ when $X$ is smooth and projective \cite[Section 1]{DK}.
Note that $E(X_\C, 1,1) = \chi_\C(X)$.
We have
\[
E(Z_{Kap}(X,t)) = \prod_{p,q} \Bigl(\frac1{1-u^p v^q t} \Bigr)^{e_{p,q}(X)}
\]
(\cite{C1,C2}, see also \cite{GLM}).
Sometimes it is convenient to consider the appropriate truncation of the Hodge-Deligne polynomial
to make it invariant under birational equvalence (and even under stable birational equivalence, cf \cite{LL}, Definition 3.4).
Thus we consider
\[
\Psi_X(t) := E(X_\C, -t, 0) \in \Z[t].
\]
If $X$ is smooth and projective, then
\be{eq-psi}
\Psi_X(t) = \sum_{p \ge 0} h^{p,0}(X_\C) t^p. 
\ee

\item {\bf Hodge realization:} $k \subset \C$. We can encode more information about the Hodge structure than in the Hodge polynomials
by considering the full Hodge realization
\[
\mu_{Hdg}: K_0(Var/k) \to K_0(HS),
\]
where $K_0(HS)$ denotes the Grothendieck ring of polarizable pure rational Hodge structures.
For a smooth projective $X$, $\mu_{Hdg}(X) = [H^*(X_\C,\Q)]$.
This gives rise to a well-defined realization using the main result of \cite{Bit}.

Note that the Hodge polynomial $E$ is the composition of the Hodge realization and the
natural homomorphism
\[
K_0(HS) \to \Z[u,v]
\]
which maps a pure Hodge structure $\HH$ to $\sum_{p,q} (-1)^{p+q} h^{p,q}(\HH)$.

For a smooth projective variety, $H^*(Sym^k(X),\Q)$ is a pure Hodge structure
isomorphic to $Sym^k(H^*(X,\Q))$.
This implies that the homomorphism $\mu_{Hdg}$ is compatible
with taking symmetric powers.


\end{itemize}

%
%
%
%
%
%
%
%

\subsection{The Grothendieck ring and rationality questions}

In this section $k$ is a field of characteristic zero.

\begin{lemma}
Let $X$, $X'$ be smooth birationally equivalent varieties.
Then we have an equality in the Grothendieck ring $K_0(Var/k)$:
\[
[X'] - [X] = \L \cdot \MM
\]
where $\MM$ is a linear combination of classes of smooth projective varieties of dimension $d-2$.
\end{lemma}
\begin{proof}
The Weak Factorization Theorem \cite{W,AKMW} says that $X'$ can be obtained from $X$
using a finite number of blow ups and blow downs with smooth centers;
thus to prove the theorem we may assume $X' = Bl_Z(X)$ where
$Z$ is a smooth subvariety of $X$ of codimension $c \ge 2$.

In this case we have
\[\;
[X'] - [X] = [\P(N_{Z/X})] - [Z] = ([\P^{c-1}] - 1) \cdot [Z] = \L \cdot [\P^{c-2} \times Z].
\]
\end{proof}

\begin{corollary}\label{cor-ratM}
If $X$ is a rational smooth $d$-dimensional variety, then 
\[
[X] = [\P^d] + \L \cdot \MM_X
\]
where $\MM_X$ is a linear combination of classes of smooth projective varieties of dimension $d-2$.
\end{corollary}

We are led to the following definition:

\begin{definition}\label{def-defect}
Let $X/k$ be an irreducible $d$-dimensional variety.
We call the class
\[
\MM_X := \frac{[X] - [\P^d]}{\L} \in K_0(Var/k)[\L^{-1}]
\]
the rational defect of $X$.
\end{definition}

\begin{example} \label{ex-hypers}
Let $X/k$ be a smooth hypersurface of dimension $d$. 
Then by the Weak Lefschetz theorem there is an isomorphism of Hodge structures
\[
H^*(X, \Q) = H^*(\P^d, \Q) \oplus H^d(X, \Q)^{prim},
\]
and by construction the Hodge realization of the rational defect $\MM_X$ is the
Hodge structrue of weight $(d-2)$ obtained by the Tate twist of $H^d(X, \Q)^{prim}$:
\[
\mu_{Hdg}(\MM_X) = [H^d(X, \Q)^{prim}(1)] \in K_0(HS).
\]
\end{example}

In our study of rationality of cubics the rational defect $\MM_X$ 
is an analog of the intermediate Jacobian considered by Clemens-Griffiths \cite{CG}
and the Clemens-Griffiths component of the derived category introduced by Kuznetsov 
(see \cite{KuzCG1,KuzCG2,BBS,Kuz13}).

By Corollary \ref{cor-ratM}, if $X$ is smooth and rational, then
the rational defect $\MM_X$ can be lifted to an element of $K_0(Var/k)$.
A weak version of the converse statement follows from Theorem \ref{thmLL} below.

\medskip

We recall that contrary to our intuition it is an open question whether
$[X] = [Y]$ implies that $X$ and $Y$ are birationally equivalent \cite{Grom}, \cite[Question 1.2]{LL}, \cite{LiS}, \cite{LaS}, \cite{Lit}.

There is however the following powerful result due to Larsen and Lunts.
Recall that two smooth projective varieties $X$ and $Y$ are called stably birationally equivalent if for some $m, n \ge 1$
$X \times \P^m$ is birationally equivalent to $Y \times \P^n$.


\begin{theorem}\cite{LL}\label{thmLL}
Let $k$ be a field of characteristic zero. 
The quotient-ring $K_0(Var/k) \big/ (\L)$ is naturally isomorphic to the free abelian group
generated by classes of stable birational equivalence of smooth projective connected varieties together with its natural ring structure.

In particular, if $X$ and $Y_1, \dots Y_m$ are smooth projective connected varieties
and
\[
[X] \equiv \sum_{j=1}^m n_j [Y_j] \; (mod \; \L), 
\]
for some $n_j \in \Z$, then $X$ is stably birationally equivalent to one of the $Y_j$.
\end{theorem}
\begin{proof}
See \cite[Proposition 2.7, Corollary 2.6]{LL}.
Larsen and Lunts have $k=\C$. However, their proof only
relies on the Weak Factorization Theorem which holds true for any field of characteristic zero \cite{W,AKMW}.
See also \cite{Bit}.
\end{proof}

In general stable birational equivalence is weaker than birational equivalence \cite{BCSS}.
However, for varieties of non-negative Kodaira dimension these two notions
coincide:

\begin{lemma}\label{lemma-mrc}
If $X$ and $Y$ are stably birationally equivalent varieties of the same dimension.
If $X$ is not uniruled then $X$ and $Y$ are birational.
\end{lemma}
\begin{proof}
This fact follows from the existence of the MRC fibration \cite{KMM, Kol-MRC, GHS}.
See \cite[Corollary 1]{LiS} for a slightly different proof.
\end{proof}


\subsection{The Cancellation conjecture}

In Section \ref{sec-rational} we rely on the following Conjecture:

\begin{conjecture}\label{conj-L}
$\L$ is not a zero divisor in $K_0(Var/k)$.
\end{conjecture}

At the moment this is not known for any field $k$.
Validity of this conjecture has been discussed
in \cite[3.3]{DL}, \cite{LiS}, \cite{LaS} and \cite[Conjecture 14]{Lit}.

Note that it is known that $K_0(Var/k)$ is not an integral domain \cite{P}, \cite{Kol}.

\section{The Hilbert scheme of two points}
\label{sec-hilb}

In this section $X/k$ is a reduced quasi-projective variety with all irreducible components of dimension $d$. 
One defines $\Hilb{X} = Hilb_2(X)$ as the Hilbert scheme of subschemes of $X$ of length two.
$\Hilb{X}$ admits an open subvariety $X^{[2],0}$ parametrizing reduced length two subschemes,
i.e. pairs of distinct $k$-rational or Galois conjugate points on $X$
defined over a quadratic extension of $k$.

Thus we have an isomorphism
\[
X^{[2],0} \simeq \Sym{X} - X.
\]
Points of the closed complement of $X^{[2],0}$ parametrize points on $X$ together with a tangent direction.

It is well-known that if $X$ is smooth, then $\Hilb{X}$ is also smooth and has a presentation
\[
\Hilb{X} \simeq \frac{Bl_\Delta(X \times X)}{\Z/2}
\]
where $\Delta \subset X \times X$ is the diagonal, and
the action of $\Z/2$ on the blow up is induced by swapping the two factors.

In general for each $p \ge 0$, we introduce a locally closed subvariety $Sing(X)_p \subset X$ and
a closed subvariety $Sing(X)_{\ge p} \subset X$:
\[\bal
Sing(X)_p &= \{x \in X: \dim T_{x,X} = d + p\}, \\
Sing(X)_{\ge p} &= \{x \in X: \dim T_{x,X} \ge d + p\}.
\eal\]
We endow these subvarieties with the reduced
subscheme structure in the case $X$ is non-reduced.
We have
\[\bal
Sing(X)_0 &= X - Sing(X)\\
Sing(X)_{\ge 1} &= Sing(X). \\
\eal\]

On each stratum $Sing(X)_p$ the tangent sheaf $T_X$ restricts to a sheaf $\TT_p$
of constant fiber dimension; thus $\TT_{p}$ is locally free 
of rank $d + p$ \cite[Exercise II.5.8]{Har}.

\begin{lemma}\label{lemma-HilbX}
The complement $\Hilb{X} - X^{[2],0}$ admits a stratification by locally closed subvarieties $Z_p$, $p \ge d$ with
\[
Z_p \simeq \P_{Sing(X)_p}(\TT_p).
\] 
\end{lemma}
\begin{proof}
$\Hilb{X} - \Hilb{X}_0$ parametrizes non-reduced subschemes of length $2$ on $X$; we let $Z_p$ to denote
the locus of those subschemes whose support is contained in $Sing(X)_p$. The natural morphism
$Z_p \to Sing(X)_p$ is the projectivization of $\TT_p$.
\end{proof}

\begin{corollary}\label{cor-HilbX}
1) We have the following formula in $K_0(Var/k)$:
\be{sym-hilb-general}
[\Hilb{X}] = [\Sym{X}] + ([\P^{d-1}] - 1) [X] + \sum_{q \ge 1} \L^{d+q-1} \cdot [Sing(X)_{\ge q}].
\ee

2) In particular, if $X$ is a hypersurface in a smooth variety $V$ of dimension $d+1$, 
then 
\be{sym-hilb}
[\Hilb{X}] = [\Sym{X}] + ([\P^{d-1}] - 1) [X] + \L^d \cdot [Sing(X)].
\ee
\end{corollary}
\begin{proof}
1)
We make a straightforward computation based on Lemma \ref{lemma-HilbX}:
\[\bal\;
[\Hilb{X}] &= ([\Sym{X}] - [X]) + \sum_{p \ge 0} [Z_p] = \\
&= ([\Sym{X}] - [X]) + \sum_{p \ge 0} ([\P^{d-1}] + \sum_{q=1}^{p} \L^{d+q-1}) \cdot [Sing(X)_p] = \\
&= [\Sym{X}] + ([\P^{d-1}] - 1) [X] + \sum_{p \ge 0} \sum_{q=1}^{p} \L^{d+q-1} \cdot [Sing(X)_p] = \\
&= [\Sym{X}] + ([\P^{d-1}] - 1) [X] +  \sum_{q \ge 1} \sum_{p \ge q} \L^{d+q-1} \cdot [Sing(X)_p] = \\
&= [\Sym{X}] + ([\P^{d-1}] - 1) [X] +  \sum_{q \ge 1} \L^{d+q-1} \cdot [Sing(X)_{\ge q}].\\
\eal\]

2) As $T_{x,X} \subset T_{x,V}$ and $V$ is smooth, we have $\dim T_{x,X} \le d+1$, thus $Sing(X)_p = \emptyset$ for $p > 1$
and (\ref{sym-hilb}) follows from (\ref{sym-hilb-general}).
\end{proof}

Finally we need the following Lemma:

\begin{lemma} \label{lemma-psi-hilb}
Let $X$ be a complex smooth projective variety.
The number of holomorphic one and two-forms of $\Hilb{X}$ are given by
\[\bal
h^{1,0}(\Hilb{X}) &= h^{1,0}(X)\\
h^{2,0}(\Hilb{X}) &= h^{2,0}(X) + \frac{h^{1,0}(X)(h^{1,0}(X)-1)}{2}. \\
\eal\]
\end{lemma}
\begin{proof}
We have $H^{p,0}(\Hilb{X}) \simeq
\Bigl(Sym^2(H^*(X))\Bigr)^{p,0}$.
In particular,
\[\bal
H^{1,0}(\Hilb{X}) \simeq \Bigl(Sym^2(H^*(X))\Bigr)^{1,0} &= H^{1,0}(X)   \\
H^{2,0}(\Hilb{X}) \simeq \Bigl(Sym^2(H^*(X))\Bigr)^{2,0} &=  H^{2,0}(X)  \oplus \Lambda^2 H^{1,0}(X). \\
\eal\]
\end{proof}

\section{The Fano variety of lines on a cubic}
\label{sec-fano}

\subsection{Definition and basic properties}
\label{subsec-fano-def}

In this section $k$ is an arbitrary field, $Y$ a cubic $d$-fold in $\P^{d+1} = \P(V)$, $\dim_k(V) = d+2$.
Let the equation of $Y$ be
\[
G \in \Gamma(\P^{d+1}, \OO(3)) = Sym^3(V^*).
\]
We allow $Y$ to have arbitrary singularities.

We consider the Grassmannian $Gr(2, V)$ of lines on $\P^{d+1}$ and its universal
rank two bundle $U \subset \OO_{Gr(2, d+2)} \otimes V$.
The section $G$ gives rise to a section $\wt{G} \in \Gamma(Gr(2,d+2),Sym^3(U^*))$.
One defines the Fano \emph{scheme} of lines on $Y$ as the zero locus of this
section
\be{Fano-scheme}
Z(\wt{G}) \subset Gr(2,V).
\ee

The Fano scheme could have non-reduced components (see \cite[Remark 1.20 (i)]{AK}).
In this paper we ignore the nonreduced structure of $Z(\wt{G})$ and let
\[
F(Y) := Z(\wt{G})_{red} \subset Gr(2, V)
\]
to be the Fano variety.

\medskip

$F(Y)$ is connected if $d \ge 3$.
If we don't assume $Y$ to be smooth, $F(Y)$ may be singular or reducible.
If $Y$ is smooth, then $F(Y)$ is smooth of dimension $2d-4$,
and $F(Y)$ is irreducible if $d \ge 3$.
This is a particular case of the following:

\begin{proposition}
If $Y$ is non-singular along a line $L \subset Y$, then
$L$ represents a smooth point of $F(Y)$ on an irreducible component
of codimension $2d-4$.
\end{proposition}
\begin{proof}
See \cite{AK, BV}. 
\end{proof}

Finally we recall that in the smooth case the canonical class of $F(Y)$ is
given by
\be{can-FY}
\omega_{F(Y)} = \OO(4-d)
\ee
where $\OO(1)$ is induced from the Pl\"ucker embedding
\[
F(Y) \subset Gr(2,d+2) \subset \P^{\binom{d+2}{2}-1}. 
\]
Thus we see that
$K_{F(Y)} \ge 0$ if and only if $d \le 4$.

\subsection{Decomposability}
\label{subsec-decomp}

In this section $k$ is a field of characteristic zero and $Y/k$ is a smooth cubic $d$-fold. 

For our study of rationality of cubics in Section \ref{sec-rational}
the following property of the Fano variety of lines will be relevant:

\begin{definition}\label{def-decomp} 
Let $W$ be an irreducible $2k$-dimensional variety. We call $W$ {decomposable} (resp. {stably decomposable})
if $W$ is birationally equivalent (resp. stably birationally equivalent) to either $V \times V'$ or $\Hilb{V}$
for some $k$-dimensional varieties $V$, $V'$.
\end{definition}

As $char(k) = 0$, we may resolve singularities and so we will assume that $V$ and $V'$ are smooth and projective.

If $W$ is not uniruled, then by Lemma \ref{lemma-mrc} stable decomposability is the same
as decomposability. This applies in particular to the Fano variety $F(Y)$ of
smooth cubic threefolds and fourfolds, as by (\ref{can-FY}) $F(Y)$ have non-negative
canonical class for $d=3,4$, and hence are not uniruled.

Furthermore in dimensions $d=3$ and $d=4$ we can effectively solve the question of decomposability.
Below we show that cubic threefolds have indecomposable Fano variety and
for cubic fourfolds $F(Y)$ can be only decomposed as $\Hilb{S}$ where $S$ is a $K3$ surface.
We rely in particular on the Hodge theoretic considerations, most notably
on the $\Psi$-polynomial (\ref{eq-psi}) which is non-zero in these dimensions.

For $d \ge 5$ stable decomposability and decomposability of the Fano variety of lines $F(Y)$
are potentially different as $F(Y)$ has negative canonical class, so in particular is uniruled and even
rationally connected. The $\Psi$-polynomial of $F(Y)$ is zero for $d \ge 5$,
so that Hodge numbers do not give any control on decomposability.

\begin{proposition}\label{prop-indec3}
For a smooth cubic threefold $Y/k$ the Fano variety $F(Y)$ is not (stably) decomposable.
\end{proposition}

\begin{proof}
It is sufficient to prove the claim over the algebraic closure of $k$.
By the Lefschetz principle we may assume $k = \C$.

We need to show that the surface $F(Y)$ is not birationally equivalent
to $C \times C'$ or $\Sym{C}$ where $C$ and $C'$ are smooth projective curves.

The $\Psi$-polynomial (\ref{eq-psi}) of the Fano surface $F(Y)$ 
is
\[
\Psi_{F(Y)}(t) = 1 + 5t + 10t^2 
\]
(see Example \ref{ex-F3}).

On the other hand the $\Psi$-polynomial of $C \times C'$ is equal to
\[
(1+g(C)t)(1+g(C')t). 
\]
As $\Psi_{F(Y)}(t)$ does not admit a non-trivial integer (or even real) factorization,
$F(Y)$ is not birational to $C \times C'$.

By Lemma \ref{lemma-psi-hilb}, the $\Psi$-polynomial of $\Sym{C}$ is 
\[
1 + g(C)\,t + \frac{g(C)(g(C)-1)}{2}\,t^2,
\]
thus if $F(Y)$ is birational to $\Sym{C}$, then $g(C) = 5$.

By Lemma \ref{lemma-SymC}, $\Sym{C}$ is a minimal surface.
Minimal models for surfaces of general type are unique,
so that if $F(Y)$ and $\Sym{C}$ are birationally equivalent, then there exists a morphism
\[
F(Y) \to \Sym{C} 
\]
which is a composition of contractions of $(-1)$-curves, and this can
only happen if
\[
h^{1,1}(F(Y)) > h^{1,1}(\Sym{C}).
\]

This is a contradiction since $h^{1,1}(F(Y)) = 25$ (Example \ref{ex-F3}), 
$h^{1,1}(\Sym{C}) = g(C)^2 + 1 = 26$ (this can be shown as in the proof of Lemma \ref{lemma-psi-hilb}).
\end{proof}

\begin{lemma}\label{lemma-SymC}
Let $C$ be a complex smooth projective curve of genus $g > 0$.
If $\Gamma \subset \Sym{C}$ is a smooth rational curve, 
then
\[
\deg \Gamma^2 = 1-g. 
\]
In particular $\Sym{C}$ does not contain $(-1)$-curves unless $g=2$.
\end{lemma}
\begin{proof}
In what follows we identify points of $\Sym{C}$ with effective degree two divisors
on $C$, in particular $\Gamma$ is a $1$-parameter family of such divisors.

We first show that $\Gamma$ parametrizes fibers of a degree $2$ covering $C \to \P^1$
so that $C$ is necessarily a hyperelliptic curve.

Fix a point $c \in C$. The Abel-Jacobi map 
\[\bal
\Sym{C} &\to Jac(C)  \\
c_1 + c_2 &\mapsto \OO(c_1 + c_2 - 2c)
\eal\]
contracts the rational curve $\Gamma$.
Hence all degree two divisors parametrized by $\Gamma$ are rationally equivalent.
Let $\LL$ denote the corresponding complete linear system.

We have $\Gamma \subset \LL$ and it is easy to see that in fact $\Gamma = \LL$: otherwise the surface 
$\Sym{C}$ would contain a linear projective subspace
$|\LL|$ of dimension $> 1$. Finally, $\LL$ has no fixed components: if 
$c_1 + c_2$ and $c_1' + c_2'$ are rationally equivalent divisors and $c_1 = c_1'$, 
then two points $c_2$ and $c_2'$ are rationally equivalent; since $g(C) > 0$ 
this can only happen if $c_2 = c_2'$.

Thus we have shown that $\Gamma$ corresponds to a complete linear system of degree $2$ and dimension $1$,
which gives rise to a $2:1$ covering $C \to \P^1$ and the associated hyperelliptic involution $\tau: C \to C$.
In these terms:
\[
\Gamma = \bigl\{ \{ x, \tau(x)\}, x \in C \bigr\}. 
\]

\medskip

Let $\pi: C^2 \to \Sym{C}$ denote the natural degree two covering.
Consider the preimage $\wt{\Gamma} = \pi^*(\Gamma)$ of $\Gamma$ in $C^2$.
Using the projection formula and the fact that 
$\pi^*$ is multiplicative we get
\[
2 \cdot \Gamma^2 = \pi_* \pi^*(\Gamma^2) = \pi_* \wt{\Gamma}^2,
\]
and after taking degrees we obtain
\be{pi-1}
\deg \Gamma^2 = \frac12 \deg \wt{\Gamma}^2.
\ee

As $\wt{\Gamma} = \{(x, \tau(x)): x \in C\}$ is the image of the diagonal $\Delta \subset C^2$
under the automorphism $id \times \tau$, we have
\be{pi-2}
\deg \wt{\Gamma}^2 = \deg \Delta^2 = \deg c_1(C) = 2-2g.
\ee

The claim follows from (\ref{pi-1}) and (\ref{pi-2}).

\end{proof}

\subsection{Decomposability and the associated $K3$ surface for cubic fourfolds}
\label{subsec-decomp4}

\begin{proposition} \label{prop-indec4} 
If $Y/k$ is a smooth cubic fourfold, and $F(Y)$ is (stably) decomposable, then $F(Y)$ is birationally equivalent to a Hilbert scheme of two points
on a $K3$ surface.
\end{proposition}
\begin{proof}
The proof is similar to that of Proposition \ref{prop-indec3}.
The $\Psi$-polynomial of the Fano variety is given by:
\[
\Psi_{F(Y)}(t) = 1 + t^2 + t^4
\]
(see Example \ref{ex-F4}).

If $F(Y)$ is birationally equivalent to a product of two surfaces $S$,$S'$, then
\[
1 + t^2 + t^4 = \Psi_S(t) \cdot \Psi_{S'}(t),
\]
$\deg \Psi_S(t), \deg \Psi_{S'}(t) \le 2$.

However, as $1 + t^2 + t^4$ does not admit a non-trivial factorization
into a product of two polynomials with positive integer coefficients, 
such decomposition is not possible.
 
Assume now that $F(Y)$ is birationally equivalent to $\Hilb{S}$ for a
smooth projective surface $S$. 
Let $q = h^{1,0}(S)$, $p_g = h^{2,0}(S)$.
By Lemma \ref{lemma-psi-hilb} for $\Psi_{\Hilb{S}}$ to match $\Psi_{F(Y)}$ we must have $q = 0$, $p_g = 1$.

It can be proved directly or applying \cite[Corollary 1]{AA}, that
$S$ has Kodaira dimension $\kappa_S = 0$.
We may replace $S$ by its minimal model; so we assume $S$ is a minimal surface.
Thus by the Enriques-Kodaira classification of surfaces $S$ can be a $K3$, an abelian surface,
an Enriques surface or a hyperelliptic surface.
Among these four types of surfaces $q=0$, $p_g = 1$ only holds for a $K3$ surface.
\end{proof}

The following two definitions are given by Hassett \cite{Has2}:

\begin{definition}
A smooth complex cubic fourfold $Y$ is called special if the rank of the sublattice of algebraic
cycles in $H^4(Y, \Z)$ is at least two.
\end{definition}

Note that the sublattice of algebraic classes in $H^4(Y, \Z)$ coincides
with the Hodge lattice $H^{2,2}(Y, \C) \cap H^4(Y,\Z)$ since the integral Hodge conjecture
is known for cubic fourfolds \cite{Z,Mur77,V}.

\begin{definition}\label{def-assK3}
A polarized $K3$ surface $(S, H)$ is associated to $Y$ if for some algebraic cycle $T \in H^4(Y, \Z)$ we have
a Hodge isometry between primitive Hodge lattices
\[
\left<h^2, T\right>^\perp \subset H^4(Y, \Z)(1) 
\]
and
\[
\left<H\right>^\perp \subset H^2(S, \Z)
\]
with its Beauville-Bogomolov form \cite{B}.
\end{definition}

\begin{proposition}\label{prop-spec4}
Let $Y/\C$ be a smooth cubic fourfold.
If the Fano variety $F(Y)$ is decomposable, then $Y$ is special and
the $K3$ surface from Proposition \ref{prop-indec4} is associated to $Y$ in terms of Hodge structure.
\end{proposition}
\begin{proof}

By Proposition \ref{prop-indec4}, $F(Y)$ is birationally equivalent to $\Hilb{S}$ for
a $K3$ surface $S$.
Now the result follows e.g. from \cite[Theorem 2]{Ad}.

%
\end{proof}

\begin{remark} The condition of decomposability of $F(Y)$ is strictly stronger
than the condition of $Y$ having an associated $K3$ surface \cite{Ad}. 
Seventy-four is the smallest discriminant for which a special cubic fourfold has an associated $K3$ surface
but $F(Y)$ is generically
not birational to a Hilbert scheme \cite{Ad}.
\end{remark}

\begin{corollary}\label{cor-indec4}
For a very general smooth cubic fourfold $Y/\C$ the Fano variety $F(Y)$ is not decomposable.
\end{corollary}
\begin{proof}
Special cubic fourfolds form a countable union of divisors in the moduli space of all cubics \cite{Has2}.
\end{proof}

\section{The $Y$-$F(Y)$ relation}
\label{sec-formula}

\subsection{The relation in $K_0(Var/k)$}

\begin{theorem} \label{main-thm1} Let $Y$ be a reduced cubic hypersurface of dimension $d$.
We have the following relations in $K_0(Var/k)$:
\be{main-formula-hilb}
[\Hilb{Y}] = [\P^d] [Y] + \L^2 [F(Y)]
\ee
\be{main-formula-sym}
[\Sym{Y}] = (1+\L^d) [Y] + \L^2 [F(Y)] - \L^d[Sing(Y)] 
\ee
where $Sing(Y)$ is the singular locus of $Y$.
\end{theorem}
\begin{proof}
We first note that (\ref{main-formula-sym}) follows from (\ref{main-formula-hilb}) using
the formula (\ref{sym-hilb}) of Corollary \ref{cor-HilbX}.

Let us now prove (\ref{main-formula-hilb}).
Consider the incidence variety 
\[
W := \{(x \in L): L \subset \P^{d+1}, x \in Y\}.
\]
In other words $W$ is the projectivization of the vector bundle
$T_{\P^{d+1}}\big|_{Y}$ on $Y$.

Let 
\be{phi}
\phi: \Hilb{Y} \dashrightarrow W := \{(x \in L): L \subset \P^{d+1}, x \in Y\}
\ee
%
be a rational morphism which is defined as follows.
A point $\tau \in \Hilb{Y}$ corresponds to a length $2$ subscheme of $Y$:
$\tau$ can be a pair of $k$-rational points,
a $k$-rational point with a tangent direction, or a pair of Galois conjugate points.

In any case there is a unique $k$-rational line $L = L_\tau$ passing through $\tau$.
For general $\tau$, the intersection $\xi = L \cap Y$ is a length $3$ scheme and there is
a third $k$-rational intersection point $x \in L \cap Y$.
We define $\phi(\tau) := (x \in L_\tau)$.

In fact $\phi$ is a birational isomorphism and $\phi^{-1}$ is defined by
mapping $(x \in L)$ to the subscheme of length $2$ obtained as the residual 
intersection of $L$ with $Y$.

\medskip

The morphism $\phi$ fits into the following diagram
\[\xymatrix{
U \ar@{^{(}->}[d] \ar[rr]^\simeq & & U' \ar@{^{(}->}[d] \\
\Hilb{Y}  \ar@{-->}[rr]^\phi & & W  \\
Z \ar@{^{(}->}[u] \ar[dr]^q & & Z' \ar@{^{(}->}[u] \ar[dl]_{q'} \\
& F(Y) &
}\]

Here $Z \subset \Hilb{Y}$ is the closed subvariety consisting of those $\tau \in \Hilb{Y}$
that the corresponding line $L_\tau$ is contained in $Y$
and $Z' \subset W$ is the closed subvariety consisting of $(x \in L)$ with $L$ is contained in $Y$.
$U$ and $U'$ are the open complements to $Z$ and $Z'$ respectively.

Note that $W$ is a $\P^d$-bundle over $Y$.
Furthermore $q': Z' \to F(Y)$ is a $\P^1$-bundle over $F(Y)$ and
similarly, $q: Z \to F(Y)$ is a $Sym^2(\P^1) = \P^2$-bundle over $F(Y)$.
Thus the Fiber Bundle Formula (\ref{proj-bun}) implies that 
\[\bal\;
[W] &= [\P^d][Y] \\
[Z] &= [\P^2][F(Y)] \\
[Z'] &= [\P^1][F(Y)]. \\
\eal\]

Putting everything together we obtain:
\[
[\Hilb{Y}] - [\P^2][F(Y)] = [\P^d][Y] -  [\P^1][F(Y)] 
\]
or equivalently
\[
[\Hilb{Y}] = [\P^d][Y] + \L^2[F(Y)].
\]
\end{proof}


\subsection{Examples and immediate applications}

\begin{corollary}\label{chiFY}
1) Let $Y$ be a cubic hypersurface over an arbitrary field. Then for the etale
Euler characteristic we have
\[
\chi(F(Y)) = \frac{\chi(Y) (\chi(Y) - 3)}{2} + \chi(Sing(Y)).
\]

2) Let $Y$ be a real cubic hypersurface. Then
\[
\chi_\R(F(Y)) = \left\{\begin{array}{ll}
\frac12(\chi_\R(Y)^2 + \chi_\C(Y)) - \chi_\R(Sing(Y)), & d \; odd \\
\\
\frac12\bigl(\chi_\R(Y)(\chi_\R(Y) - 4) + \chi_\C(Y)\bigr) + \chi_\R(Sing(Y)), & d \; even \\
\end{array}\right.
\]


3) Let $k = \F_q$, the finite field and let $N_1 = \#Y(\F_q)$, $N_2 = \#Y(\F_{q^2}), N_s = \#Sing(Y)(\F_q)$.
Then
\[
\#F(Y)(\F_q) = \frac{{N_1}^2 - 2(1+q^d)N_1 + N_2}{2q^d} + q^{d-2} N_s.
\]

\end{corollary}
\begin{proof}
The formulas follow by applying respective realization homomorphisms
from Section \ref{sec-grothendieck}
to the formula for $\Sym{Y}$ of Theorem \ref{main-thm1} and formulas (\ref{numSym2}), (\ref{chiSym2}), (\ref{chiRSym2})
using that $\#(\L^k) = q^k$, $\chi_\C(\L^k) = 1$, $\chi_\R(\L^k) = (-1)^k$.


\end{proof}



\begin{example}\label{ex-duVal}
Let $k$ be an algebraically closed field.
Let $Y/k$ be a cubic surface with $r$ isolated du Val singularities
with the sum of Milnor numbers of the singular points equal to $n$.

Let $\wt{Y}$ be the minimal desingularization of $Y$. We know that $\chi(\wt{Y}) = 9$
(for example because $\wt{Y}$ is a smooth rational surface
with $K_{\wt{Y}}^2 = K_Y^2 = 3$).
On the other hand we have
$
\chi(\wt{Y}) = \chi(Y) + n
$
so that
\[
\chi(Y) = 9 - n 
\]
and Corollary \ref{chiFY} (1) implies that
\be{linesSurf}
\#F(Y)(k) = \chi(F(Y)) = \frac{(9-n)(6-n)}{2} + r
\ee	

In particular, a smooth cubic surface has 27 lines and a cubic surface with an ordinary
double point has 21 lines.

The same formula (\ref{linesSurf}) has been obtained in \cite{BW}, page 255,
over the complex numbers using case by case analysis (see also \cite[Table 9.1]{Dol});
another proof is given in \cite[Satz 1.1]{KM}.
\end{example}

\begin{example}\label{ex-real}
The structure of real cubic surfaces and the number of lines on them is a classical
subject initiated by Schl\"afli \cite{Sch} in 1863. In particular Schl\"afli proves that
there can be $3$, $7$, $15$ or $27$ lines on a smooth real cubic surface.
See \cite{Seg42, KM,Sil,Kol-surf,PT} for a discussion on real cubic surfaces.

Recall that a smooth real cubic surface $Y$ is either isomorphic
to a blow of $\R\P^2$ in $k$ pairs of complex conjugate points and $6-2k$ real points,
where $k$ is $0, 1, 2$ or $3$ in which case $Y$ is a rational surface
or is an irrational surface with $Y(\R)$ homeomorphic to a disjoint union of $\R\P^2$ and a two-sphere $\S^2$.

If $Y$ is a smooth rational real cubic surface, then
\[
\chi_\R(Y) = \chi_\R(\R\P^2) + (6-2k) \chi_\R(\L) = 1 - (6-2k) = 2k - 5
\]
and Corollary \ref{chiFY} (2) implies that there are
\[
\# F(Y)(\R) = \chi_\R (F(Y)) = \frac{(2k-5)(2k-9)+9}{2} = 2k^2 - 14k + 27
= \left\{
\begin{array}{ll}
27, & k = 0 \\
15, & k = 1 \\
7, & k = 2 \\
3, & k = 3 \\
\end{array}\right.
\]
real lines on $Y$. 

In the case of irrational $Y$ we have 
\[
\chi_\R(Y) = \chi_\R(\R\P^2) + \chi_\R(\S^2) = 1 + 2 = 3
\]
and thus there are
\[
\# F(Y)(\R) = \chi_\R (F(Y)) = \frac{-3+9}{2} = 3
\]
real lines on $Y$.

Similarly one can deduce a more general formula for the number of
lines on a real cubic surface with du Val singularities which has been 
also computed in \cite[Satz 2.8]{KM}.
\end{example}

\begin{example}
Let  $Y$ be a cone over a $(d-1)$-dimensional cubic $\bar{Y}$.
Then we have
\[\bal\;
[Y] &= 1 + \L \cdot [\bar{Y}]\\ 
Sym^2 [Y] &= 1 + \L \cdot [\bar{Y}] + \L^2 \cdot Sym^2 [\bar{Y}]\\ 
[Sing(Y)] &= 1 + \L \cdot [Sing(\bar{Y})]\\ 
[F(Y)] &= [\bar{Y}] + \L^2 [F(\bar{Y})]\\ 
\eal\]
(for the last equality note that the set of lines on $Y$ that pass through the
vertex of the cone are parametrized by $\bar{Y}$, whereas the rest of the lines
project isomorphically onto a line on $\bar{Y}$ and the fiber of $Y$ over each
such line is a $2$-plane contained in $Y$).

In this case one can see that the formula of Theorem \ref{main-thm1} (2) for $\bar{Y}$ implies
the same kind of formula for $Y$.
\end{example}

\subsection{The relation in $K_0(Var/k)[\L^{-1}$]}



\begin{theorem}\label{main-thm2}
Let $Y$ be a cubic $d$-fold and let $\MM_Y$ be its rational defect (see Definition \ref{def-defect}).
There is the following relation in $K_0(Var/k)[\L^{-1}]$:
\[\;
[F(Y)] = Sym^2(\MM_Y + [\P^{d-2}]) - \L^{d-2}(1 - [Sing(Y)]). 
\]
\end{theorem}
\begin{proof}
We compute the symmetric square of
\[
[Y] = [\P^d] + \L \cdot \MM_Y \in K_0(Var/k)[\L^{-1}]
\]
using identities (\ref{symsum}), (\ref{symL}):
\[
[\Sym{Y}] = Sym^2[\P^d] + \L \cdot [\P^d]  \cdot \MM_Y + \L^2 \cdot Sym^2(\MM_Y).
\]

Substituting this into Theorem \ref{main-thm1}(2) gives:
\be{eqL2}\bal
\L^2 \cdot [F(Y)] &= [\Sym{Y}] - (1+ \L^d)[Y] + \L^d \cdot [Sing(Y)] = \\
&= \L^2 \cdot Sym^2(\MM_Y) + \L^2 \cdot [\P^{d-2}] \cdot \MM_Y + \\
&+ \Bigl(Sym^2[\P^d] -(1+\L^d)[\P^d]\Bigr)  + \L^d \cdot [Sing(Y)].\\
\eal\ee

Finally it is easy to see that in fact
\[
Sym^2[\P^d] -(1+\L^d)[\P^d] = \L^2 \cdot Sym^2([\P^{d-2}]) - \L^d 
\]
and we get the claim dividing (\ref{eqL2}) by $\L^2$.

\end{proof}

\begin{corollary}\label{cor-main-thm2}
There is the following relation in $K_0(Var/k)[\L^{-1}]$:
\[
[F(Y)] = Sym^2(\MM_Y) + [\P^{d-2}] \cdot \MM_Y + \sum_{k=0}^{2d-4} a_k \L^k + \L^{d-2} \cdot [Sing(Y)]
\]
where 
\[
a_k = \left\{
\begin{array}{cc}
\left[\frac{k+2}{2}\right], & k < d-2 \\
\left[\frac{d-2}{2}\right], & k = d-2 \\
\left[\frac{2d-2-k}{2}\right], & k > d-2 \\
\end{array}\right.
\]
\end{corollary}

\medskip

We illustrate how Theorem \ref{main-thm2} allows to compute the class $[F(Y)]$
in terms of the rational defect $\MM_Y$ in two examples of rational cubics.

\medskip

\begin{example}\label{ex-sing}

Let $Y$ is a cubic hypersurface with a single ordinary
double point $P$ over an arbitrary field.
Projecting from the point $P$ one establishes an isomorphism
\[
Bl_P(Y) \simeq Bl_V(\P^d),
\]
where $V \subset \P^d$ is a smooth complete
intersection of a cubic with a quadric.

We find the rational defect of $Y$.
Let $E$ be the exceptional divisor of $Bl_{P}(Y)$. We have:
\[
[Y] - [P] + [E] = [Bl_P(Y)] = [Bl_V(\P^d)] = [\P^d] + \L [V]
\]
so that
\[
[Y] = [\P^d] + \L [V] - ([E] - 1).
\]

The exceptional divisor $E$ is a $(d-1)$-dimensional quadric, so that
$[E] = [\P^{d-1}]$ for $d$ even and $[E] = [\P^{d-1}] + \L^\frac{d-1}{2}$ for $d$ odd.
This leads to the following formula:
\[
[Y] = \left\{\begin{array}{ll}
\, [\P^d] + \L ([V] - [\P^{d-2}]), & d \; \text{even} \\
\, [\P^d] + \L ([V] - [\P^{d-2}] - \L^\frac{d-3}{2}), & d \; \text{odd}  \\
\end{array}\right.
\]
and for the rational defect of $Y$ we get
\[
\MM_Y = \left\{\begin{array}{ll}
\, [V] - [\P^{d-2}], & d \; \text{even} \\
\, [V] - [\P^{d-2}] - \L^\frac{d-3}{2}, & d \; \text{odd}  \\
\end{array}\right.
\]

The two varieties $F(Y)$ and $\Sym{V}$ are 
known to be birational: $V$ parametrizes lines passing through $P$
and for two such lines there is a residual line in the plane spanned
by the two lines \cite{CG}. 
Now we can find the class of the Fano variety in $K_0(Var/k)[\L^{-1}]$ using Theorem \ref{main-thm2}:
\be{FYnode}
[F(Y)] = \left\{\begin{array}{ll}
\, Sym^2([V]), & d \; \text{even} \\
\, Sym^2([V] - \L^\frac{d-3}{2}), & d \; \text{odd}  \\
\end{array}\right.
\ee


\medskip

If $d=2$, and $k$ is algebraically closed, then $V$ consists of six isolated points, and there are $Sym^2(6) = 21$
lines on a cubic surface with one node in accordance with Example \ref{ex-duVal}.

If $d=3$, $C := V$ is a canonically embedded genus $4$ curve and we have
\[\bal\;
[F(Y)] &= Sym^2([C] - 1) = \\
&= Sym^2([C]) - Sym^2(1) - ([C] - 1) = \\
&= Sym^2([C])  -  [C].
\eal\]

It is known \cite{CG} that in this case the birational morphism $Sym^2(C) \to F(Y)$ 
glues two disjoint copies of $C$ together.

If $d=4$, $S := V$ is a $K3$-surface and
\[
[F(Y)] = [Sym^2(S)]
\]
which agrees with \cite[Lemma 6.3.1]{Has2}.

 
\end{example}

\begin{example}
Let $d$ be even and assume that $Y$ is a smooth cubic $d$-fold containing two disjoint $d/2$-planes $P_1, P_2$.
In this case $Y$ is rational as there is a birational map $P_1 \times P_2 \to Y$
mapping $(a,b) \in P_1 \times P_2$ to the third point of intersection
of the line $L_{a,b}$ through $a$ and $b$ with $Y$.

Resolving indeterminacy locus of this map and its inverse we find as isomorphism
\[
Bl_{P_1, P_2} (Y) \simeq Bl_{V} (P_1 \times P_2).
\]
Here $V$ is a $(d-2)$-dimensional variety consisting of points $(a,b)$ such that $L_{a,b} \subset V$.
In fact $V$ is a complete intersection of divisors $(1,2)$ and $(2,1)$ in $P_1 \times P_2$.

$V$ can also be considered as a subvariety in $F(Y)$ parametrizing lines intersecting both $P_1$ and $P_2$.

For the rational defect of $Y$ we have
\[\bal\;
\MM_Y &= \MM_{P_1 \times P_2} +  [V] - 2 [\P^{d/2}][\P^{d/2-2}]\\
\eal\]
unless $d=2$ in which case the third term disappears, $V$ is a set of $5$ points and $\MM_Y = 6$.
For $d=4$, $S:=V$ is a $K3$ surface and 
$\MM_Y = [S] + [\P^1]^2 - 2[\P^2]$.

In this example $F(Y)$ is again birationally equivalent to $\Sym{V}$: two lines $L_1$, $L_2$ on $Y$ intersecting
both $P_1$ and $P_2$ and which are generic with this property,
determine a smooth cubic surface $T$ in their span $\langle L_1, L_2 \rangle$.
$T$ is equipped with two more lines $E_1 = P_1 \cap T$, $E_2 = P_2 \cap T$.
There is a unique line $L$ on $T$ which does not intersect the quadrilateral formed by $L_1, L_2, E_1, E_2$.

One can see that the assignment $\{L_1, L_2\} \mapsto L$ defines a birational morphism $\Sym{V} \to F(Y)$.
For the inverse map, starting with a generic line $L \subset Y$ not intersecting $P_1$, $P_2$
one finds the unique $3$-plane containing $L$ and intersecting $P_1$ and $P_2$ in some
lines $E_1$, $E_2$.
The intersection of this $3$-plane with $Y$ is a smooth cubic surface $T$, and one finds
a unique pair of skew lines $L_1$, $L_2$ intersecting both $E_1$, $E_2$ and not intersecting $L$.

Theorem \ref{main-thm2} gives an expression of $[F(Y)]$ in terms of $[V]$, which will be of the form
\[
[F(Y)] = [\Hilb{V}] + \L \cdot (\dots).
\]
The term in brackets is a certain combination of classes $\L^i$ and $\L^j \cdot [V]$.
For $d=4$ this last term vanishes and we simply get
\[
[F(Y)] = [\Hilb{S}] \in K_0(Var/k)[\L^{-1}]. 
\]
However, as Hassett remarks in \cite{Has2} for $d=4$
these two varieties are not isomorphic 
(see \cite[Section 6.1]{Has3} for details).
\end{example}

\section{Hodge structure of the Fano variety $F(Y)$}
\label{sec-hodge}

In this section we assume $Y$ to be a smooth complex cubic $d$-fold.
We compute the Hodge structure of the Fano variety of lines $F(Y)$ in terms
of the Hodge structure of $Y$.

By the Weak Lefschetz theorem there is the following decomposition of Hodge structure of $Y$:
\[
H^*(Y, \Q) = \bigoplus_{k=0}^d \Q(-k) \oplus H^d(Y, \Q)^{prim},
\]
where $H^d(Y, \Q)^{prim}$ is the primitive cohomology with respect to the hyperplane section.
We have
\[
H^d(Y, \Q) = \left\{
\begin{array}{cc}
H^d(Y, \Q)^{prim}, & d \; odd \\
H^d(Y, \Q)^{prim} \oplus \Q(-\frac{d}{2}), & d \; even \\
\end{array}\right.
\]

The Hodge structure of $F(Y)$ is expressed in terms of 
the weight $(d-2)$ Hodge structure
\be{HHY}
\HH_Y := H^d(Y, \Q)^{prim}(1).
\ee

\begin{theorem}\label{thm-hodge} Let $Y$ be a smooth complex cubic hypersurface of dimension $d$.
There is the following decomposition for the Hodge structure of the Fano variety of lines on $Y$:
\[
H^*(F(Y), \Q) \simeq Sym^2(\HH_Y) \oplus \bigoplus_{k=0}^{d-2} \HH_Y(-k) \oplus \bigoplus_{k=0}^{2d-4} \Q(-k)^{a_k}
\]
and
\[
a_k = \left\{
\begin{array}{cc}
\left[\frac{k+2}{2}\right], & k < d-2 \\
\left[\frac{d-2}{2}\right], & k = d-2 \\
\left[\frac{2d-2-k}{2}\right], & k > d-2 \\
\end{array}\right.
\]

In particular, if $d$ is even, then all odd-dimensional cohomology of $F(Y)$ vanish.

\end{theorem}

Note that for $d \ge 3$ the first interesting (i.e. non-Tate) cohomology group of $F(Y)$ is
\[
H^{d-2}(F(Y), \Q) = \left\{
\begin{array}{cc}
\HH_Y, & d \; odd \\
\HH_Y \oplus \Q(-\frac{d-2}{2})^{[\frac{d+2}{4}]}, & d \; even \\
\end{array}\right.
\]

\begin{proof}
Consider the Hodge realization homomorphism 
\[
\mu_{Hdg}: K_0(Var/\C) \to K_0(HS/\Q).
\]
$\mu_{Hdg}$ maps the Tate class $\L^p=[\A^p]$ to the class of the Hodge-Tate structure $[\Q(-p)]$ of weight $2p$,
which is invertible;
this implies that $\mu_{Hdg}$ descends to a well-defined ring homomorphism 
\[
K_0(Var/\C)[\L^{-1}] \to K_0(HS)
\]
which we will also denote by $\mu_{Hdg}$.

\medskip

Definition \ref{def-defect} of the rational defect $\MM_Y$ is compatible with the definition of $\HH_Y$:
\[
[\HH_Y] = \mu_{Hdg}(\MM_Y) \in K_0(HS),
\]
see Example \ref{ex-hypers}.

Applying the realization $\mu_{Hdg}$ to the decomposition of Corollary \ref{cor-main-thm2}
we get
\[
[H^*(F(Y), \Q)] = [Sym^2(\HH_Y)] + \bigl[\bigoplus_{k=0}^{d-2} \HH_Y(-k)\bigr] + [\bigoplus_{k=0}^{2d-4} \Q(-k)^{a_k}]. 
\]

It is well-known that the category of polarizable Hodge structures is semisimple \cite[Corollary 2.12]{PS},
in particular if two polarizable pure Hodge structures $H_1$ and $H_2$ have equal
classes in the Grothendieck ring, then $H_1$ and $H_2$ are isomorphic. Thus we obtain
\[
H^*(F(Y), \Q) \simeq Sym^2(\HH_Y) \oplus \bigoplus_{k=0}^{d-2} \HH_Y(-k) \oplus \bigoplus_{k=0}^{2d-4} \Q(-k)^{a_k}. 
\]

\end{proof}

In principle Theorem \ref{thm-hodge} allows to compute all the Hodge numbers of the Fano variety
$F(Y)$ of a smooth
cubic $d$-fold using the following Lemma:

\begin{lemma}
The primitive Hodge numbers $h^{d-q,q}$ of a smooth complex cubic $d$-fold $Y$ are contained
in the range $\frac{d-1}{3} \le q \le \frac{2d+1}{3}$ and for those $q$ are 
given as follows:
\[
h^{d-q,q}_{prim}(Y) = \binom{d+2}{3q-d+1}. 
\]


\end{lemma}
\begin{proof}
The proof is a standard computation of Hodge numbers of a smooth hypersurface based
on the work of Griffiths \cite{Grif}.
\end{proof}

\begin{example}\label{ex-F3}
If $d = 3$, $H^3(Y,\Q)$ has weight $3$ and Hodge numbers $(0,5,5,0)$;
thus $\HH_Y$ has weight one with $h^{1,0} = h^{0,1} = 5$
and Theorem \ref{thm-hodge} gives a decomposition
of the Hodge structure of the Fano surface:
%
\[\left[
\begin{array}{c|ccccc|c}
H^4 & & & 1 & &       &   \Q(-2) \\
H^3 & & 5 & & 5 &     &   \HH_Y(-1) \\
H^2 & 10 && 25 && 10    &   Sym^2(\HH_Y) \\
H^1 & & 5 & & 5 &     &   \HH_Y \\
H^0 & & & 1 & &       &   \Q \\
\end{array}\right]
\]

This result has been known since the work of Clemens and Griffiths \cite{CG}.
\end{example}

\begin{example}\label{ex-F4}
If $d = 4$, $H^4(Y,\Q)$ has weight four and Hodge numbers $(0,1,21,1,0)$,
$\HH_Y$ has weight two with $h^{2,0} = h^{0,2} = 1$ and $h^{1,1} = 20$ 
(primitive classes are of codimension one in $H^{2,2}(Y)$)
and we get a decomposition for the Hodge structure of the
hyperk\"ahler fourfold $F(Y)$:
\[\left[
\begin{array}{c|ccccc|c}
H^8 & & & 1 & &       &   \Q(-4) \\
H^6 & & 1 & 21 & 1 &     &   \HH_Y(-2) \oplus \Q(-3) \\
H^4 & 1 & 21 & 232 & 21 & 1    &   Sym^2(\HH_Y) \oplus \HH_Y(-1) \oplus \Q(-2) \\
H^2 & & 1 & 21 & 1 &     &   \HH_Y \oplus \Q(-1) \\
H^0 & & & 1 & &       &   \Q \\
\end{array}\right]
\]

This has been deduced by Beauville and Donagi \cite[Proposition 2]{BD}
from the fact that $F(Y)$ is deformation equivalent to the Hilbert scheme
of two points on a $K3$ surface.
\end{example}

\section{Rational cubic hypersurfaces}
\label{sec-rational}

In this section $k$ is a field of characteristic zero.
In addition we assume that Conjecture \ref{conj-L} is true for $k$.

\begin{theorem} \label{thm-rat} 
Let $Y$ be a smooth cubic hypersurface of dimension $d \ge 3$
over a field satisfying Conjecture \ref{conj-L}.
If $Y$ is rational, then $F(Y)$ is stably decomposable in the sense of Definition \ref{def-decomp}.
\end{theorem}
\begin{proof}
By Corollary \ref{cor-ratM} we have
\[
[Y] = [\P^d] + \L \cdot \MM_Y,
\]
where $\MM_Y \in K_0(Var/k)$ is a combination
of classes of smooth projective varieties of dimension equal to $d-2$:
\be{def-rat-cubic}
\MM_Y = \sum_{i=1}^{m} [V_i] - \sum_{j=1}^{n} [W_j].
\ee

Since we assume that $\L$ is not a zero-divisor,
the formula in Theorem \ref{main-thm2} is valid in $K_0(Var/k)$:
\[
[F(Y)] = Sym^2(\MM_Y + \P^{d-2}) - \L^{d-2}.
\]

We set $V_{m+1} := \P^{d-2}$ and compute using (\ref{symsum}):
\[\bal\;
Sym^2(\MM_Y + \P^{d-2}) &= Sym^2 \Bigl( \sum_{i=1}^{m+1} [V_i] - \sum_{j=1}^{n} [W_j] \Bigr) = \\
&= Sym^2 \Bigl( \sum_{i=1}^{m+1} [V_i] \Bigr) - Sym^2 \Bigl( \sum_{j=1}^{n} [W_j] \Bigr)- \\
&- \Bigl(  \sum_{i=1}^{m+1} [V_i] - \sum_{j=1}^{n} [W_j] \Bigr)  \Bigl(  \sum_{j=1}^{n} [W_j] \Bigr)
 = \\
&= \sum_{i=1}^{m+1} Sym^2 [V_i] + \sum_{1 \le i < i' \le m+1} [V_i][V_{i'}] + \sum_{1 \le j \le j' \le n} [W_j][W_{j'}] - \\
&-\Bigl( \sum_{j=1}^{n} Sym^2 [W_j] 
+ \sum_{i=1}^{m+1} \sum_{j=1}^{n} [V_i][W_j] \Bigr).\\
\eal\]

By Corollary \ref{cor-HilbX}, $\Sym{V}$ and $\Hilb{V}$ are congruent modulo $\L$.
Thus reducing modulo $\L$ we obtain
\be{FYdecomp}\bal\;
[F(Y)] &\equiv \sum_{i=1}^{m+1} [\Hilb{V_i}] + \sum_{1 \le i < i' \le m+1} [V_i][V_{i'}] + \sum_{1 \le j \le j' \le n} [W_j][W_{j'}] - \\
&-\Bigl( \sum_{j =1}^n [\Hilb{W_j}] +  \sum_{i=1}^{m+1} \sum_{j=1}^{n} [V_i][W_j] \Bigr) \; (mod \; \L) \\
\eal\ee

The result now follows from Theorem \ref{thmLL}.
\end{proof}

\begin{remark}
It can be seen from the proof of Theorem \ref{main-thm2}, that for Theorem \ref{thm-rat} 
to hold in dimension $d$, instead of relying on the general Cancellation Conjecture \ref{conj-L}
it is sufficient to assume that $\L^2$ does not annihilate combinations of classes
of varieties of dimension $\le 2(d-2)$.

It would even suffice if we knew for 
combinations $\alpha \in K_0(Var/k)$
of classes of varieties of dimension $\le 2(d-2)$:
\[
\L^2 \cdot \alpha = 0 \implies \alpha \in \L\cdot K_0(Var/k).
\]

\end{remark}

\begin{remark}
If $k$ is algebraically closed, then for the rational defect (\ref{def-rat-cubic})
we have $n=m$ as 
\[
1 = h^{1,1}(Y) = h^{1,1}(\P^d) + h^{0,0}(\MM_Y) = 1 + m-n.
\]

This makes the number of the product terms in the right-hand-side of
(\ref{FYdecomp}) balanced,
and leaves an extra $\Hilb{V_i}$ term.
It is then very likely that in the decomposition (\ref{FYdecomp}) the class of the Fano variety will in fact match one of the $\Hilb{V_i}$,
and not one of the products $V_i \times V_{i'}$, $W_j \times W_{j'}$. 
\end{remark}

\medskip

Using the results on indecomposability of the Fano variety obtained in Sections \ref{subsec-decomp}, \ref{subsec-decomp4}
we can make Theorem \ref{thm-rat} very useful in dimensions $d=3,4$.

\begin{theorem}\label{thm-rat3}
Let $k$ be a field satisfying Conjecture \ref{conj-L}.
Any smooth cubic threefold $Y/k$ is irrational.
\end{theorem}
\begin{proof}
Follows from Theorem \ref{thm-rat} and Proposition \ref{prop-indec3}.
\end{proof}

\begin{theorem}\label{thm-rat4}
1) Let $Y/k$ be a smooth cubic fourfold over a field $k$ satisfying Conjecture \ref{conj-L}.
If $Y$ is rational, then the Fano variety $F(Y)$ is birational to $\Hilb{S}$ for a $K3$ surface $S/k$.

2) If Conjecture \ref{conj-L} is true for $k=\C$, then
a very general\footnote{``Very general" means: lying in the complement of a 
countable union of divisors in the moduli space.} smooth complex cubic fourfold is irrational.
\end{theorem}
\begin{proof}
Follows from Theorem \ref{thm-rat}, Proposition \ref{prop-indec4} and Colorollary \ref{cor-indec4}.
\end{proof}

\providecommand{\arxiv}[1]{\href{http://arxiv.org/abs/#1}{\tt arXiv:#1}}


\address{
{\bf Sergey Galkin}\\
National Research University Higher School of Economics (HSE)\\
Faculty of Mathematics and Laboratory of Algebraic Geometry\\
7 Vavilova Str. \\
117312, Moscow, Russia\\
e-mail: {\tt Sergey.Galkin@phystech.edu}
}

\medskip

\address{
{\bf Evgeny Shinder}\\
College of Mathematics \\
University of Edinburgh \\
The Kings Buildings \\
Mayfield Road \\
Edinburgh EH6 3JL, UK\\
e-mail: {\tt E.Shinder@ed.ac.uk}
}

\end{document}